\documentclass[11pt]{article}

\usepackage{amsmath,amssymb,amsthm}%,latexsym
\usepackage{color}
\usepackage{graphicx}

\newtheorem{theorem}{Theorem}

\newtheorem{definition}{Definition}

\newtheorem{cor}{Corollary}
\newtheorem{proposition}{Proposition}
\newtheorem{remark}{Remark}

\def \beq{ \begin{equation}}
\def \eeq{\end{equation}}

\title{Mass independent 
shapes for relative equilibria\\
in the two dimensional constant positive \\curved three body problem
}

\date{}  %to erase the date stamp
\begin{document}

\maketitle

\author{\centerline{Toshiaki~Fujiwara$^1$ and Ernesto~P\'erez-Chavela$^2$} 
$^1$ College of Liberal Arts and Sciences, Kitasato University, 1-15-1 Kitasato, Sagamihara, 252-0329, Kanagawa, Japan. 

$^2$ Department of Mathematics, Instituto Tecnol\'ogico Aut\'onomo de M\'exico (ITAM), R\'io Hondo 1, M\'exico City, 01080, Mexico.
}

\begin{abstract}
In the planar three-body problem under Newtonian potential,
it is well known that any masses, located at the vertices of an equilateral triangle generates a relative equilibrium, known as the  
Lagrange relative equilibrium.
In fact,
the equilateral triangle
is the unique mass independent shape for a relative equilibrium
in this problem.

The two dimensional positive curved three-body problem, is a natural extension of the Newtonian three-body problem to the sphere 
$\mathbb{S}^2$, where the masses are moving under the influence of the cotangent potential. 
S.~Zhu showed that in 
this problem, equilateral triangle on a rotating meridian 
can form a relative equilibria for any masses. 
This was the first report of mass independent shape on $\mathbb{S}^2$ 
which can form a relative equilibrium.

In this paper, we  show that, in addition to the equilateral triangle,
there exists one isosceles triangle 
on a rotating meridian, with two equal angles seen from the centre of
$\mathbb{S}^2$ given by $2^{-1}\arccos((\sqrt{2}-1)/2)$, which always form a relative equilibrium for any choice of the masses.
Additionally we prove that, the equilateral and the above isosceles relative equilibrium are unique with this characteristic.
We also prove that
each relative equilibrium generated by a mass independent shape
is not isolated from the other relative equilibria.
\end{abstract}


{\bf Math. Subject Class 2020:} 70F07, 70F10, 70F15

%\tableofcontents

\section{Introduction}
The two dimensional positive curved three-body problem has been studied 
for several authors, for instance  
\cite{Borisov1,Borisov2,Diacu-EPC1,Diacu-EPC2,Diacu3,
EPC1,EPC2,M-S,tibboel,zhu,zhu2}.
In all 
these papers the masses are moving under the influence of the cotangent potential, which is the natural extension of the planar Newtonian problem to the sphere.

Consider a point $q$ on $\mathbb{S}^2$, $|q|^2=1$.
In the spherical coordinates, $q$ is represented by 
$q=(\sin\theta\cos\phi,\sin\theta\sin\phi,\cos\theta)\in \mathbb{R}^3$.
The Lagrangian for the three body problem on $\mathbb{S}^2$
is given by
\begin{equation}\label{theLagrangian}
L = K - V,
\end{equation}
where the kinetic energy $K$ and the cotangent potential $V$ are 
\begin{equation*}
K=\sum\nolimits_k\frac{m_k}{2}
		\left(\dot{\theta}_k^2+\sin^2(\theta_k)\dot{\phi}_k^2\right),
\quad 
V= - \sum\nolimits_{i<j}
	\frac{m_i m_j\cos\sigma_{ij}}{\sqrt{1-\cos^2(\sigma_{ij})}}.
\end{equation*}
Dot on symbols represents the time derivative,
the indexes $i,j,k$ run for $1,2,3$,
and $m_k$ are the masses.
The angle $\sigma_{ij}$
is the angle between the two points $q_i$ and $q_j$
as seen from the centre of $\mathbb{S}^2$. 
In order to avoid the singularities \cite{Diacu-EPC2}, the range of 
$\sigma_{ij}$ is restricted to $0< \sigma_{ij}< \pi \,\,
\mbox{ for all } \,\, i\ne j.$ Then
$\cos\sigma_{ij}$ is given by the inner product of $q_i$ and $q_j$,
namely
\begin{equation}
\label{fundamentalrelation}
\cos\sigma_{ij}
=\cos\theta_i\cos\theta_j+\sin\theta_i\sin\theta_j\cos(\phi_i-\phi_j).
\end{equation}

The equations of motion derived from Lagrangian \eqref{theLagrangian} are
\begin{equation}
\begin{split}
&\frac{d}{dt}\left(m_i \dot\theta_i\right)
=m_i \sin\theta_i\cos\theta_i \dot\phi_i^2
	- \frac{\partial V}{\partial \theta_i},\\
&\frac{d}{dt}\left(m_i\sin^2(\theta_i)\dot\phi_i\right)
	= - \frac{\partial V}{\partial \phi_i}.
\end{split}
\end{equation}

\begin{definition}\label{def:RE}
A relative equilibrium ($RE$ in short)
on $\mathbb{S}^2$
is a solution of the equations of motion
with $\dot\theta_i=0$ and $\dot\phi_i=\omega=$ constant.
\end{definition}

Then, the equations of motion for a relative equilibrium
are reduced to
\begin{equation}\label{eqOfMotForTheta}
\begin{split}
\omega^2 m_i \sin\theta_i\cos\theta_i
=\sum_{j\ne i}\frac{m_i m_j}{\sin^3(\sigma_{ij})}
	\Big(
		\sin\theta_i\cos\theta_j-\cos\theta_i\sin\theta_j\cos(\phi_i-\phi_j)
	\Big)
\end{split}
\end{equation}
and
\begin{equation}\label{eqforphi}
\begin{split}
&m_1 m_2\sin\theta_1 \sin\theta_2 \sin(\phi_1-\phi_2)/\sin^3(\sigma_{12})\\
=&m_2 m_3\sin\theta_2 \sin\theta_3 \sin(\phi_2-\phi_3)/\sin^3(\sigma_{23})\\
=&m_3 m_1\sin\theta_3 \sin\theta_1 \sin(\phi_3-\phi_1)/\sin^3(\sigma_{31}).
\end{split}
\end{equation}
We call them as ``equations of relative equilibria'' in the following.

\begin{remark}
One might be interested in $RE$ with $\theta_1=\theta_2=\theta_3$.
The three bodies move along the same circle, which is parallel to the equator. This group of $RE$ must be the simplest $RE$.
In the corresponding Newtonian problem,
$RE$ with $r_1=r_2=r_3$ are realised only when $m_1=m_2=m_3$
(and the shape is equilateral), 
they form a ``choreography''.

However, Diacu and Zhu  showed that
besides the equilateral $RE$ with $m_1=m_2=m_3$ (which form a ``choreography''),
a nontrivial group of isosceles $RE$ exist on $\mathbb{S}^2$: for $m_2=m_3$ and $m_1=m_3\nu$ with $\nu\in (0,2)$ \cite{zhu2}.
The latter solution is not ``choreography'',
because ``choreography'' requires ``equal time spacing on the orbit'' between the bodies. The latter solutions have different time spacing if $\nu\ne 1$.
This is an interesting example for
motions along a same orbit with different time spacing between bodies.
Examples of non-circular choreographies on $\mathbb{S}^2$
were studied by  Montanelli\cite{Montanelli}.
\end{remark}

\begin{definition}[Euler and Lagrange]\label{EulerAndLagrange}
An Euler $RE$ ($ERE$) is an $RE$ where three bodies are on the same geodesic. 
If this is not the case, we call it Lagrange $RE$ ($LRE$).
\end{definition}

In \cite{F-P}, we showed that there are two cases of $ERE$:
three bodies are on a rotating meridian, or they are on the equator. Obviously for $ERE$ the rotation axis is the $z$-axis.

\begin{definition}[Shape and configuration]\label{shape-conf}
A shape is the set $\{\sigma_{12},\sigma_{23},\sigma_{31}\}$,
and a configuration is the set 
$\{\theta_1,\theta_2,\theta_3,\phi_1-\phi_2,\phi_2-\phi_3,\phi_3-\phi_1\}$.
\end{definition}

In the following,
we simply write
$\{\sigma_{ij}\}$ and
$\{\theta_k, \phi_i-\phi_j\}$
for the above sets.
Similarly, we write
$\{m_k\}$ for $\{m_1,m_2,m_3\}$.

In 1772 \cite{L}, J.L. Lagrange published an amazing result for the Newtonian problem of the three bodies: {\it Any three arbitrary masses located at the vertices of an equilateral triangle, generates a relative equilibria}. In other words, if we put any three different masses, as for instance the Sun, Jupiter and a small stone, at the vertices of an equilateral triangle, then there exists an angular velocity $\omega$ such that the three masses rotate uniformly around their center of mass, the motion is like a rigid body. 
This is that we call {\it mass independent shape for $RE$}. In a natural way we extend this concept to the sphere, and we rise the question: Are there mass independent shape $RE$ for the three body problem on $\mathbb{S}^2$?

It is easy to show that mass independent shape for relative equilibria of the $2$--body problem on the sphere, are all shapes $\sigma_{12} \in (0,\pi)$ except for $\sigma_{12} = \pi/2$. Ahead  in this paper we will prove this statement (see Proposition~ \ref{2bodyMassInd}).

In \cite{Diacu-EPC1}, Diacu et al.~showed
that in order to generate a $RE$ for three masses 
located at the vertices  of
an equilateral triangle parallel to the equator, the masses should be equal. In \cite{zhu}, 
Zhu proved that for the cotangent potential, any three masses 
placed on the equilateral triangle
on a rotating meridian generate a $RE$. Some years later in \cite{F-P}, we extend this result for general potentials which only depends of mutual distances among the masses. The goal of this paper is to show that, for the cotangent potential, besides to the equilateral triangle, we have one additional isosceles triangle shape on a rotating meridian, which is independent of the choice of the masses. We believe that this work will open a door for the search of new $RE$ on curved spaces and the stability of them, as well as its possible applications.

In Definition \ref{shape-conf}, 
we emphasised the difference between shape and configuration. 
In order to be more accurate, we close this section by giving the precise definitions of the concepts that we will use in this article.

\begin{definition}
A $RE$ shape for the two dimensional constant positive curved three body problem, is a shape which can form a $RE$. In particular we call Lagrange RE shape (Lagrange shape in short) and Euler $RE$ shape (Euler shape in short) to the shapes which can form $LRE$ and $ERE$ respectively.
\end{definition}
\begin{definition}[Mass independent shape for $RE$]
A mass independent $RE$ shape
is a shape $\{\sigma_{ij}\}$ that can form 
an $RE$ for any masses $\{m_k\}$.
\end{definition}

After the introduction, where we define the concepts 
that we will study here, the paper is organized as follows: in Section \ref{prelims}, we state the equations used to generate $LRE$ and $ERE$.
We prove that there are no mass independent Lagrange shapes,
and there are no mass independent Euler shapes on the equator. 
In Section \ref{main}, we prove
the main result of this article: the existence of mass independent 
Euler shape on a rotating meridian. Since the rotation axis
depends on masses, the configuration made from the mass independent shape depends on masses (see \cite{F-P} for details). At the end of this section we briefly discuss the case of the restricted three body problem on the sphere.

In Section \ref{ind-conf}, we show the configurations for different choice of the masses.
In Section \ref{cont}, we look for continuations of $RE$ shape from  
mass independent Euler shapes. 
We show that each one of them can be continued as 
a $RE$ shape which is mass dependent.
In Section \ref{Isos-eq},
we prove that isosceles or equilateral 
$RE$ shapes when 
all masses are different from each other,
are just the mass independent
shapes shown in the previous sections.
We finish the paper with an Appendix
to describe the properties of some special shapes.

\section{Preliminars}\label{prelims}
We have proved in \cite{F-P}, that the $RE$ on the sphere are solutions of the eigenvalue problem
\begin{equation}\label{eigen-pbm}
J\Psi = \lambda \Psi
\end{equation}
where the eigenvector $\Psi$ represents the direction of the rotation axis for 
a $LRE$ 
and the matrix 
$J$ is an equivalent expression of the inertia tensor given by 

\begin{equation}\label{defJ}
J=\left(\begin{array}{ccc}
m_2+m_3 & -\sqrt{m_1m_2}\cos\sigma_{12} & -\sqrt{m_1m_3}\cos\sigma_{13} \\
-\sqrt{m_2m_1}\cos\sigma_{21} &m_3+m_1& -\sqrt{m_2m_3}\cos\sigma_{23} \\
-\sqrt{m_3m_1}\cos\sigma_{31} & -\sqrt{m_3m_2}\cos\sigma_{32} & m_1+m_2
\end{array}\right).
\end{equation}

\subsection{$LRE$}

In the same paper \cite{F-P}, we prove that, for the positive curved three body problem, 
the necessary and sufficient condition for a shape to generate a $LRE$ is
$\lambda_1=\lambda_2=\lambda_3$.
Where the $\lambda_i$'s are
\begin{equation}\label{eqForLagrangeII}
\begin{split}
\lambda_1&=\frac{(m_2+m_3)\sin^3(\sigma_{23})
-m_2\cos(\sigma_{12})\sin^3(\sigma_{31})
-m_3\cos(\sigma_{31})\sin^3(\sigma_{12})}{\sin^3(\sigma_{23})},\\
\lambda_2&=\frac{(m_3+m_1)\sin^3(\sigma_{31})
-m_3\cos(\sigma_{23})\sin^3(\sigma_{12})
-m_1\cos(\sigma_{12})\sin^3(\sigma_{23})}{\sin^3(\sigma_{31})},\\
\lambda_3&=\frac{(m_1+m_2)\sin^3(\sigma_{12})
-m_1\cos(\sigma_{31})\sin^3(\sigma_{23})
-m_2\cos(\sigma_{23})\sin^3(\sigma_{31})}{\sin^3(\sigma_{12})}.
\end{split}
\end{equation}

\begin{proposition}
There are no mass independent Lagrange shapes.
\end{proposition}
\begin{proof}
The equation $\lambda_1-\lambda_2=\lambda_2-\lambda_3=0$ has the form
\begin{equation}
\left(\begin{array}{ccc}
S_{11} & S_{12} & S_{13} \\
S_{21} & S_{22} & S_{23}
\end{array}\right)
\left(\begin{array}{c}m_1 \\ m_2 \\ m_3\end{array}\right)
=0.
\end{equation}
Where, each $S_{ij}$ is a function of 
$\{\sigma_{ij}\}$.
To satisfy the above equation for any $\{m_k\}$,
all the elements $S_{ij}$ must be $0$.
But, $S_{11}=S_{12}=0$ yields
\begin{equation*}
\cos\sigma_{12}\,\frac{\sin^3(\sigma_{23})}{\sin^3(\sigma_{31})}
=\cos\sigma_{12}\,\frac{\sin^3(\sigma_{31})}{\sin^3(\sigma_{23})}
=1.
\end{equation*}
Which does not have solution for $\sigma_{ij}\in (0,\pi)$.

Therefore, there are no mass independent shapes for $LRE$.
\end{proof}

\subsection{$ERE$ on the equator}
When three bodies are on the equator, $\theta_i=\pi/2$,
and $\sin\sigma_{ij}=|\sin(\phi_i-\phi_j)|$.
Therefore the equation of motion \eqref{eqOfMotForTheta} is automatically satisfied,
and \eqref{eqforphi} takes the form
\begin{equation}\label{eqForEREOntheEquator}
\frac{m_1m_2\sin(\phi_1-\phi_2)}{|\sin(\phi_1-\phi_2)|^3}
=\frac{m_2m_3\sin(\phi_2-\phi_3)}{|\sin(\phi_2-\phi_3)|^3}
=\frac{m_3m_1\sin(\phi_3-\phi_1)}{|\sin(\phi_3-\phi_1)|^3}.
\end{equation}

\begin{proposition}
There are no mass independent Euler shapes on the equator.
\end{proposition}
\begin{proof}
Obviously, there are no mass independent solution of $\{\phi_i-\phi_j\}$.
\end{proof}

\subsection{$ERE$ on a rotating meridian}
For $ERE$ on a rotating meridian,
it is convenient to extend the range of $\theta_i$ to $-\pi\le \theta_i\le \pi$
and $\phi_i=0$.
Then 
$\sin\sigma_{ij}=|\sin(\theta_i-\theta_j)|$,
and the equations of relative equilibria \eqref{eqOfMotForTheta} are
\begin{equation}\label{eqOfMotForTheta2}
\omega^2 m_i \sin\theta_i\cos\theta_i
=\sum_{j\ne i}\frac{m_i m_j\sin(\theta_i-\theta_j)}{|\sin(\theta_i-\theta_j)|^3}.
\end{equation}

The equations of relative equilibria  \eqref{eqforphi} are automatically satisfied.

We have proved in \cite{F-P} that if $A$ defined by
\begin{equation}\label{defOfA}
A=\left(\sum_\ell m_\ell^2+2\sum_{i<j}m_im_j\cos(2(\theta_i-\theta_j))
	\right)^{1/2}
\end{equation}
is not zero, then
the necessary and sufficient condition for the shape
$\{\theta_i-\theta_j\}$
to be an Euler shape on a rotating meridian is
\begin{equation}
\label{eqThetaForMeridianForCotangent}
\begin{split}
&m_1m_2\left(\frac{s\omega^2}{2A}\sin\Big(2(\theta_1-\theta_2)\Big)
- \frac{\sin(\theta_1-\theta_2)}{|\sin(\theta_1-\theta_2)|^3}
\right)\\
=&m_2m_3\left(\frac{s\omega^2}{2A}\sin\Big(2(\theta_2-\theta_3)\Big)
- \frac{\sin(\theta_2-\theta_3)}{|\sin(\theta_2-\theta_3)|^3}
\right)\\
=&m_3m_1\left(\frac{s\omega^2}{2A}\sin\Big(2(\theta_3-\theta_1)\Big)
- \frac{\sin(\theta_3-\theta_1)}{|\sin(\theta_3-\theta_1)|^3}
\right),
\end{split}
\end{equation}
for $s=1$ or $-1$.
\begin{remark}
The parameter $s=\pm 1$ comes from the fact that, when we study $RE$ on a rotating meridian, it is convenient to enlarge the range angle $\theta_k$ to $- \pi <  \theta_k < \pi$, with $\phi_k=0$. When we study a particular configuration, the equations of relative equilibria \eqref{eqOfMotForTheta} and \eqref{eqforphi} determine the sign of $s$ (see \cite{F-P} for more details).
\end{remark}

The configuration $\{\theta_k\}$ is given by
\begin{equation}\label{translationFormula}
\begin{split}
\cos(2\theta_1)&=s A^{-1}\left(
	m_1+m_2\cos\Big(2(\theta_1-\theta_2)\Big)
				+m_3\cos\Big(2(\theta_1-\theta_3)\Big)
	\right),\\
\sin(2\theta_1)&=s A^{-1}\left(
	m_2\sin\Big(2(\theta_1-\theta_2)\Big)
				+m_3\sin\Big(2(\theta_1-\theta_3)\Big)
	\right).
\end{split}
\end{equation}
The other angles $\theta_k$ are determined by 
$\theta_k=\theta_1+(\theta_k-\theta_1)$.

For the special shapes with $A=0$,
the map from the shape to configuration,
$\{\theta_i-\theta_j\}\to \{\theta_k\}$,
is not determined uniquely.
Therefore, we have to check
whether each of such shapes can satisfy the equation of motion or not.
See Appendix.

In the next section, we will show 
the existence of mass independent Euler shapes on a rotating meridian.

\section{Mass independent shapes for $ERE$ on a rotating meridian}\label{main}
To get a mass independent shape,
any term
in the parentheses of equation
\eqref{eqThetaForMeridianForCotangent}
must be zero, namely,
\begin{equation}\label{eqForMassIndependentShape0}
\frac{s\omega^2}{2A}\sin\Big(2(\theta_i-\theta_j)\Big)
=\frac{\sin(\theta_i-\theta_j)}{|\sin(\theta_i-\theta_j)|^3}
\end{equation}
for $(i,j)=(1,2), (2,3), (3,1)$.

From now on, 
in order to facilitate the reading of the manuscript
we introduce the new variables
$\tau_k=\theta_i-\theta_j$
for $(i,j,k)=(1,2,3), (2,3,1)$, and $(3,1,2)$.
The range of $\tau_k$ is $(-\pi,\pi)$.
The relation between $\sigma_{ij}$ and $\tau_k$ is $\sigma_{ij}=|\tau_k|$ and
$\sin\sigma_{ij}=|\sin\tau_k|$.
Equation \eqref{eqForMassIndependentShape0}
in $\tau_k$ variables is,
\begin{equation}\label{eqForMassIndependentShape}
\frac{s\omega^2}{2A}\sin(2\tau_k)
=\frac{\sin\tau_k}{|\sin\tau_k|^3}.
\end{equation}

Since 
$\tau_3=\theta_1-\theta_2=-(\tau_1+\tau_2)$,
we can take $\tau_1$ and $\tau_2$ as
the independent variable to give a shape.
To avoid the singularity, $\sin\tau_k\ne 0$.
We can restrict the range of $\tau_2\in (0,\pi)$,
because 
we can rotate the system by $\pi$
around the north pole if $\tau_2<0$.
Therefore, the non-singular shapes $\{\sigma_{ij}\}$ and 
the ordered set $(\tau_1,\tau_2)$
are in correspondence one to one in the set
\begin{equation}\label{defOfUphys}
U_\textrm{phys}
=\{(\tau_1,\tau_2)|
\tau_1\in (-\pi,\pi), 
\tau_2\in (0,\pi),
\sin(\tau_1)\sin(\tau_2)\sin(\tau_1+\tau_2)\ne 0\}.
\end{equation}
This is the shape space
for $RE$ on a rotating meridian.

Since $\sin\tau_k\ne 0$, the equations \eqref{eqForMassIndependentShape}  are equivalent to
\begin{equation}
\frac{s\omega^2}{A}\cos\tau_k
=\frac{1}{|\sin\tau_k|^3}.
\end{equation}
Given that
$\cos\tau_k=0$ cannot satisfy this equation,
we assume $\cos\tau_k\ne 0$.
So, the above conditions are
\begin{equation}\label{condition}
\frac{s\omega^2}{A}
=\frac{1}{\cos(\tau_k)|\sin\tau_k|^3}.
\end{equation}
The equations for $\tau_k$ are
\begin{equation}
\cos(\tau_1)|\sin\tau_1|^3
=\cos(\tau_2)|\sin\tau_2|^3
=\cos(\tau_3)|\sin\tau_3|^3,
\end{equation}
or in variables $\tau_1,\tau_2$
\begin{equation}\label{eqRot-0}
\cos(\tau_1)|\sin(\tau_1)|^3
=\cos(\tau_2)|\sin(\tau_2)|^3
=\cos(\tau_1+\tau_2)|\sin(\tau_1+\tau_2)|^3.
\end{equation}

In order to facilitate the reading of our main theorem, we define the following curves (see Figure \ref{figThm1}) 
$$f_1(\tau_1, \tau_2) = \cos(\tau_1)|\sin(\tau_1)|^3 - \cos(\tau_2)|\sin(\tau_2)|^3,$$  $$f_2(\tau_1, \tau_2) = \cos(\tau_1)|\sin(\tau_1)|^3 - \cos(\tau_1+\tau_2)|\sin(\tau_1+\tau_2)|^3.$$
\begin{figure}
   \centering
   \includegraphics[width=7cm]{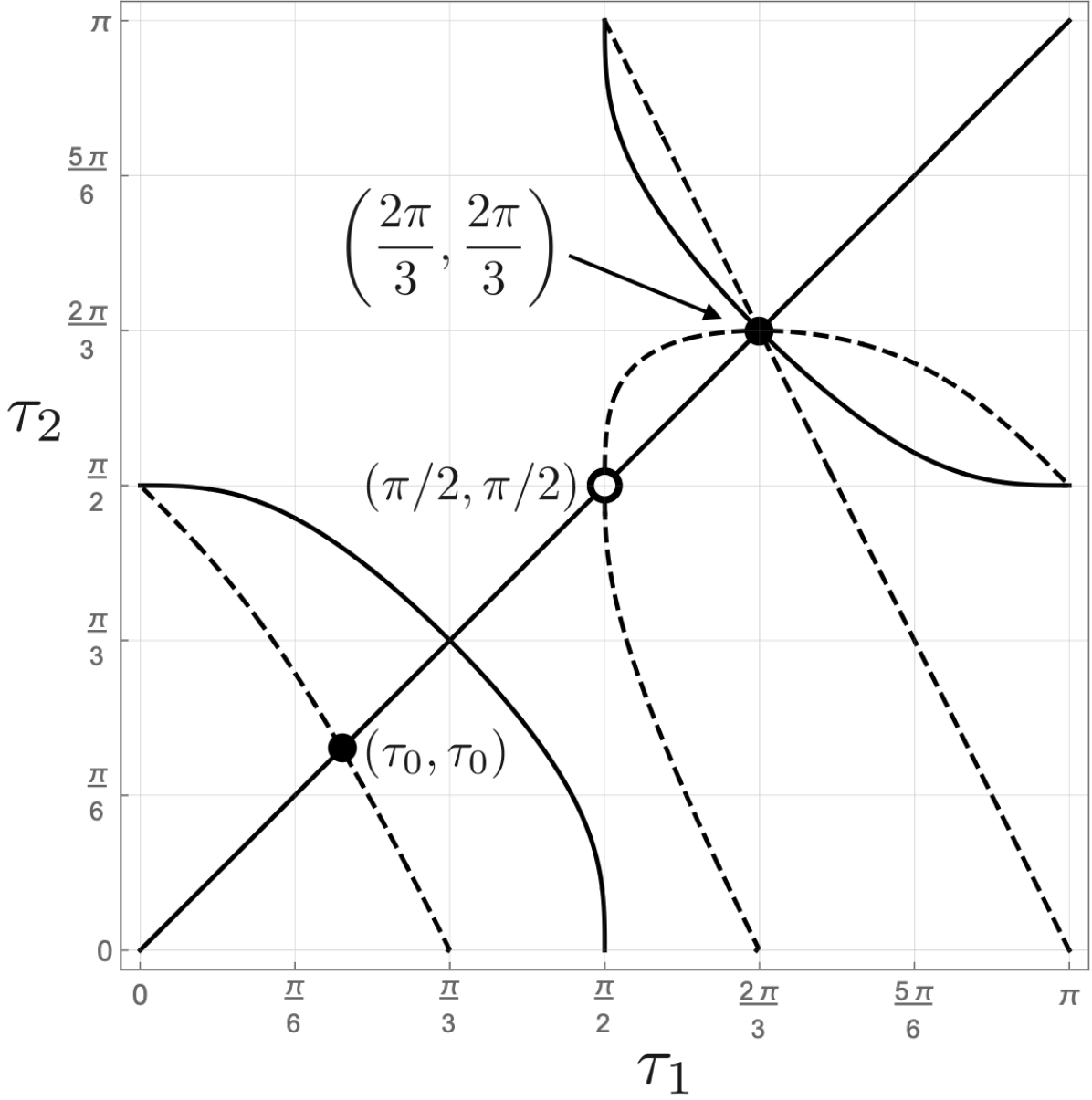}%
   \caption{The solid curves and the dashed curves represent
$f_1(\tau_1, \tau_2)=0$ and $f_2(\tau_1, \tau_2)=0$
respectively, the intersection points of these curves give us the possible mass independent shapes for $ERE$.
}
  \label{figThm1}
\end{figure}

Now, we are in conditions to state and prove the main 
result of this article.

\begin{theorem}\label{thm}
In the two dimensional constant positive curved three body problem, there are exactly two $RE$ shapes which are independent of the masses, 
both of them are on a rotating meridian (Euler shapes), one isosceles triangle with equal arc $\tau_0=2^{-1}\arccos((\sqrt{2}-1)/2)$ and one equilateral triangle with the same arc 
$\tau_\textrm{e}=2\pi/3$.
\end{theorem}

\begin{proof}
In the previous section we showed the no existence of
mass independent Lagrange shapes nor Euler shapes on the equator.
Then, it is only necessary to analyze the Euler shapes on a rotating meridian.

For an isosceles triangle on a rotating meridian,
we consider the case
$\tau=\tau_1=\tau_2 \in (0,\pi)\setminus\{\pi/2\}$.
Then, equation \eqref{eqRot-0}
takes the form
\begin{equation}
\cos(\tau)|\sin\tau|^3
=\cos(2\tau)|\sin(2\tau)|^3.
\end{equation}
Therefore, $\cos\tau$ and $\cos(2\tau)$ must have the same sign.
Namely, $0<\tau<\pi/4$ (both are positive)
or $\pi/2<\tau<3\pi/4$ (both are negative).

We divide the proof in three steps depending of the different shape of the triangle, isosceles, equilateral or scalene.

\subsubsection*{Step 1: For $0<\tau<\pi/4$ we obtain an isosceles Euler shape.}
\begin{proof}
For the case $0<\tau<\pi/4$, $\sin\tau$ and $\sin(2\tau)$ are positive. Then, the equation for $2\tau$ is
\begin{equation}
4\cos(2\tau)(\cos(2\tau)+1)=1,
\end{equation}
a second order polynomial in $\cos(2\tau)$ whose
solution is $\cos(2\tau_0)=(\sqrt{2}-1)/2$ corresponding to
$\tau_0=2^{-1}\arccos((\sqrt{2}-1)/2)=0.6810...<\pi/4$.

For this solution,
\begin{equation}\label{AforIsosceles}
A^2 =(m_1-m_2)^2+\left(3-2\sqrt{2}\right)m_1m_2
	+m_3^2+\left(\sqrt{2}-1\right)(m_2m_3+m_3m_1)
>0.
\end{equation}
Therefore, we get one isosceles solution.
\end{proof}

Observe that we didn't use any special properties for 
$\{m_k\}$ in this calculation. Any mass $m_k$ can be located in middle of the other two masses $m_i$ and $m_j$.

%the above argument holds for the mass 
%the shape $\tau_0=\theta_3-\theta_1=\theta_1-\theta_2$,
%and the shape
%$\tau_0=\theta_3-\theta_2=\theta_2-\theta_1$. 

In variables $(\tau_1,\tau_2)$, the three shapes for the isosceles $RE$ (which we are counting just as one) are
$(\tau_0,\tau_0)$, 
$(-2\tau_0,\tau_0)$, and
$(-\tau_0,2\tau_0)$.

\subsubsection*{Step 2: For $\pi/2<\tau<3\pi/4$ we obtain a unique equilateral Euler shape.}
\begin{proof}
In this case, 
the sign of $\sin\tau$ and $\sin(2\tau)$ are opposite.
Therefore the equation for $\cos(2\tau)$
is reduced to
\begin{equation}
4\cos(2\tau)(\cos(2\tau)+1)=-1.
\end{equation}
The solution in $\pi/2<\tau<3\pi/4$
is
$\cos(2\tau)=-1/2$ corresponding to 
$\tau=2\pi/3$.
Therefore,
$\tau=\theta_2-\theta_3=\theta_3-\theta_1=2\pi/3$,
and $\theta_1-\theta_2=-2\tau=-4\pi/3 \equiv 2\pi/3 \mod 2\pi$.
Namely, this is an equilateral triangle.

For this solution,
\begin{equation}
A^2=\sum_{i<j}(m_i-m_j)^2.
\end{equation}
Therefore, $A=0$ for equal masses case.
But, obviously $\theta_1-\theta_2=\theta_2-\theta_3=\theta_3-\theta_1=2\pi/3$
and $\omega^2=0$
satisfies \eqref{eqOfMotForTheta2} for equal masses case.
For not equal masses we have $A\ne 0$.
\end{proof}

To finish the proof of Theorem \ref{thm} we only have to prove the non-existence of mass independent scalene Euler shape. We will do it in the next step.

\subsubsection*{Step 3: There are no mass independent scalene 
Euler shape on a rotating meridian.}
\begin{proof}
The equation \eqref{eqRot-0}
in terms of $a=\tau_1$, $b=\tau_2$ is
\begin{equation}\label{eqRot-00}
\cos(a)|\sin(a)|^3
=\cos(b)|\sin(b)|^3
=\cos(a+b)|\sin(a+b)|^3.
\end{equation}
The use of $a$ and $b$ is for later convenience.
Obviously,
$\cos(a)$, $\cos(b)$, and $\cos(a+b)$
must have the same sign.
Therefore, there are four possible regions.
\begin{itemize}
\item[I:]$\pi/2<a<\pi,\pi/2<b<\pi$, and $\pi<a+b<3\pi/2$
\item[II:]$\pi/2<a<\pi,\pi/2<b<\pi$, and $0<a+b<\pi/2$
\item[III:]$-\pi/2<a<0$, $0<b<\pi/2$ and $0<a+b<\pi/2$
\item[IV:]$-\pi/2<a<0$, $0<b<\pi/2$ and $-\pi/2<a+b<0$
\end{itemize}

\subsubsection*{Region I: $\pi/2<a<\pi,\pi/2<b<\pi$, 
$\pi<a+b<3\pi/2$}

For this region, the equations \eqref{eqRot-00}
are equivalent to
\begin{align}
\cos(a)\sin^3(a)&=\cos(b)\sin^3(b),\label{eqReI-1}\\
2^{-1}\big(\cos(a)\sin^3(a)+\cos(b)\sin^3(b)\big)&=-\cos(a+b)\sin^3(a+b).
\label{eqReI-2}
\end{align}

Equation \eqref{eqReI-1} yields
\begin{equation}
\begin{split}\label{eqReI-1B}
0=&\cos(a)\sin^3(a)-\cos(b)\sin^3(b)\\
=&2^{-1}\sin(a-b)
	\Big(\cos(a+b)-\cos(2(a+b))\cos(a-b)\Big).
\end{split}
\end{equation}
Obviously, $a=b$ is a solution,
and this yields $a=b=2\pi/3$ as above.
Here we assume $a\ne b$
to look for other solutions.

Then equation \eqref{eqReI-1B} reduces to
\begin{equation}
\cos(a+b)
	=\cos(2(a+b))\cos(a-b).
\end{equation}
If $\cos(2(a+b))=0$, then $a+b=5\pi/4$.
But $\cos(a+b)=\cos(5\pi/4)\ne 0$.
Therefore $\cos(2(a+b))\ne 0$.
Then
\begin{equation}\label{cosDelta1}
\cos(a-b)
=\frac{\cos(a+b)}{\cos(2(a+b))}\ne 1.
\end{equation}
($=1$ is excluded, because we are assuming $a\ne b$).
The region for the absolute value of the right hand side 
is smaller than $1$, if $4\pi/3 < a+b<3\pi/2$.
 
On the other hand, using
\begin{equation*}
\begin{split}
&\cos(a)\sin^3(a)+\cos(b)\sin^3(b)\\
&=2^{-1}\sin(a+b)\Big(
	\cos(a-b)
	-\cos(2(a-b))\cos(a+b)
	\Big),
\end{split}
\end{equation*}
the equation \eqref{eqReI-2} yields
\begin{equation}
\cos(a-b)
	-\cos(2(a-b))\cos(a+b)
=-4\cos(a+b)\sin^2(a+b).
\end{equation}

Substituting \eqref{cosDelta1} into this equation, we get
the equation for $a+b$,
\begin{equation}
\frac{\cos(a+b)\sin^4(a+b)}{\cos^2(2(a+b))}
\Big(2\cos(2(a+b))+1\Big)=0.
\end{equation}
But there are no solutions for $a+b\in (4\pi/3,3\pi/2)$.

Thus, in this region the unique solution is 
the equilateral triangle shape $a=b=2\pi/3$.

\subsubsection*{Region II: $a,b,a+b\in (0,\pi/2)$}
For this region, the equations are
\begin{align}
\cos(a)\sin^3(a)&=\cos(b)\sin^3(b),\label{eqReII-1}\\
2^{-1}\big(\cos(a)\sin^3(a)+\cos(b)\sin^3(b)\big)&=\cos(a+b)\sin^3(a+b).\label{eqReII-2}
\end{align}
By the same procedure, equation \eqref{eqReII-1} yields,
$a=b$ or the same relation in \eqref{cosDelta1}.

For $a=b$, we get
 $a=b=2^{-1}\arccos((\sqrt{2}-1)/2)$
 %=0.6810...<\pi/4$
 as in the previous subsection.
 
For $a\ne b$,
we use the relation \eqref{cosDelta1}.
In the region $a+b\in (0,\pi/2)$,
the range of the solution is
$a+b\in (\pi/3,\pi/2)$.

Now, the equation \eqref{eqReII-2} yields
\begin{equation}
\cos(a-b)-\cos(2(a-b))\cos(a+b)
=4\cos(a+b)\sin^2(a+b).
\end{equation}
Substituting $\cos(a-b)$ in \eqref{cosDelta1}
into this equation, we get
\begin{equation}
\frac{\cos(a+b)\sin^2(a+b)}{\cos^2(2(a+b))}
\Big(2\cos^2(2(a+b))+\cos(2(a+b))+1\Big)=0.
\end{equation}
But, This equation has no solution.

Thus, in this region, the solution is only
the isosceles triangle shape 
$\theta_2-\theta_3=\theta_3-\theta_1=2^{-1}\arccos((\sqrt{2}-1)/2).$
Here the mass $m_3$ is placed in middle of the other two masses.

\subsubsection*{Region III: $-\pi/2<a<0$, $0<b<\pi/2$
and $0<a+b<\pi/2$}

By the definition of $a$ and $b$,
the region in terms of the coordinates 
$\theta_i$ is given by
$\theta_3-\theta_2,\theta_2-\theta_1,\theta_3-\theta_1 \in (0,\pi/2)$.
Now, we redefine 
$a=\theta_3-\theta_2$, $b=\theta_2-\theta_1$,
then $a+b=\theta_3-\theta_1$
with $a,b,a+b\in (0,\pi/2)$.

By this redefinition, equations \eqref{eqRot-00} are invariant,
and the region for the variables $a,b$ is the same as for the Region II.
Therefore, the solution is only the isosceles triangle shape
$\theta_3-\theta_2=\theta_2-\theta_1
=2^{-1}\arccos((\sqrt{2}-1)/2)$.
Here the mass $m_2$ is placed in middle of the other two masses.

\subsubsection*{Region IV: $-\pi/2<a<0$, $0<b<\pi/2$
and $-\pi/2<a+b<0$}

Using a similar argument 
for $a=\theta_3-\theta_1$, $b=\theta_1-\theta_2$ 
we obtain that
the solution is only the isosceles solution
$\theta_3-\theta_1=\theta_1-\theta_2=2^{-1}\arccos((\sqrt{2}-1)/2)$.
Here the mass $m_1$ is placed in middle of the other two masses.

This finish the proof of the step 3. \end{proof}

With all the above, we have proved Theorem \ref{thm}.
\end{proof}

\subsection{Mass independent shapes for relative equilibria in the two body problem and the
restricted three body problem on the sphere}
The arguments given in the previous section are correct, even for the case when one of the masses $m_k>0$ is really small. The next question is: What happen in the limit case when one of the mass, let's say 
$m_3 \to 0.$ 
There are two problems. One is the two body problem,
just considering the two masses $m_1$ and $m_2$
and neglecting the existence of the third mass.
Another one is
``the restricted three body problem on the sphere.'' 
where the position of $m_3$ is concerned.
This last case has been examined in \cite{Kilin, Mtz}. We will show that 
our results are still true in 
the two and the three body problem.

We start from the equations  in \eqref{eqThetaForMeridianForCotangent}.
For $m_3\to 0$,
we obtain two equations
\begin{equation}
\label{twoBodyProblem}
m_1m_2\left(\frac{s\omega^2}{2A}\sin\Big(2(\theta_1-\theta_2)\Big)
- \frac{\sin(\theta_1-\theta_2)}{|\sin(\theta_1-\theta_2)|^3}
\right)=0,
\end{equation}
and
\begin{equation}
\label{CondForTheta3InRestrictedThreeBody}
\begin{split}
&m_2\left(\frac{s\omega^2}{2A}\sin\Big(2(\theta_2-\theta_3)\Big)
- \frac{\sin(\theta_2-\theta_3)}{|\sin(\theta_2-\theta_3)|^3}
\right)\\
=&m_1\left(\frac{s\omega^2}{2A}\sin\Big(2(\theta_3-\theta_1)\Big)
- \frac{\sin(\theta_3-\theta_1)}{|\sin(\theta_3-\theta_1)|^3}
\right).
\end{split}
\end{equation}
Note that the last equation does not contain the term $m_3$.

We have the following propositions.
\begin{proposition}
\label{2bodyMassInd}
For the two body problem on a rotating meridian,
any shape $|\theta_1-\theta_2|\in (0,\pi)$ except $\pi/2$
is a mass independent shape.
\end{proposition}
\begin{proof}
For the two body problem, 
the condition of $ERE$ on a rotating meridian
is only the equation \eqref{twoBodyProblem}.
Obviously,
$|\theta_1-\theta_2|\in (0,\pi)$ except $\pi/2$
satisfies this condition by choosing $s\omega^2$ properly.
\end{proof}

\begin{proposition}
The mass independent shapes in the restricted three body
problem on a rotating meridian
are the same as in the three body problem on the sphere with finite masses.
\end{proposition}
\begin{proof}
The conditions for this problem
are \eqref{twoBodyProblem} and 
\eqref{CondForTheta3InRestrictedThreeBody}.
The last one determines the position of $m_3$.
Therefore, a mass independent shape must make
all terms inside of the parentheses zero.
That is, the conditions are the same as for the three body problem with finite masses.
\end{proof}

\section{Configuration of mass independent shape
for several masses}\label{ind-conf}
Even for a mass independent shape $\{\sigma_{ij}\}$,
the configuration $\{\theta_k, \phi_i-\phi_j\}$
depends on the masses $\{m_k\}$,
because of the mass dependence of the rotation axis
($z$-axis) given by the inertia tensor \cite{F-P}.

\subsection{Configurations of equilateral solution}
\label{secConfigEquilateral}
For not equal masses case 
(at least two masses are different),
the equilateral solution has $A>0$.
Then, by equation \eqref{condition},
$s=-1$ and $\omega^2=8A/(3\sqrt{3})$.
The configuration 
is given by equations \eqref{translationFormula},
\begin{equation}\label{conf-eq}
\begin{split}
\cos(2\theta_1)&=\frac{1}{A}\left(\frac{m_2+m_3}{2}-m_1\right),\\
\sin(2\theta_1)&=\frac{\sqrt{3}}{2A}(m_2-m_3).
\end{split}
\end{equation}
In Figure \ref{figConfigurations}, we show 
the configurations for several masses.
The configuration is uniquely determined by equation \eqref{conf-eq}.

On the other hand,
for equal masses case, the configuration is indefinite.
Namely, any configurations with $\theta_i-\theta_j=2\pi/3$,
$(i,j)=(1,2)$, $(2,3)$, $(3,1)$ satisfy the equation of motion with $\omega=0$.
(See Step 2, in the proof of Theorem \ref{thm}.)
\begin{figure}
   \centering
   \includegraphics[width=3cm]{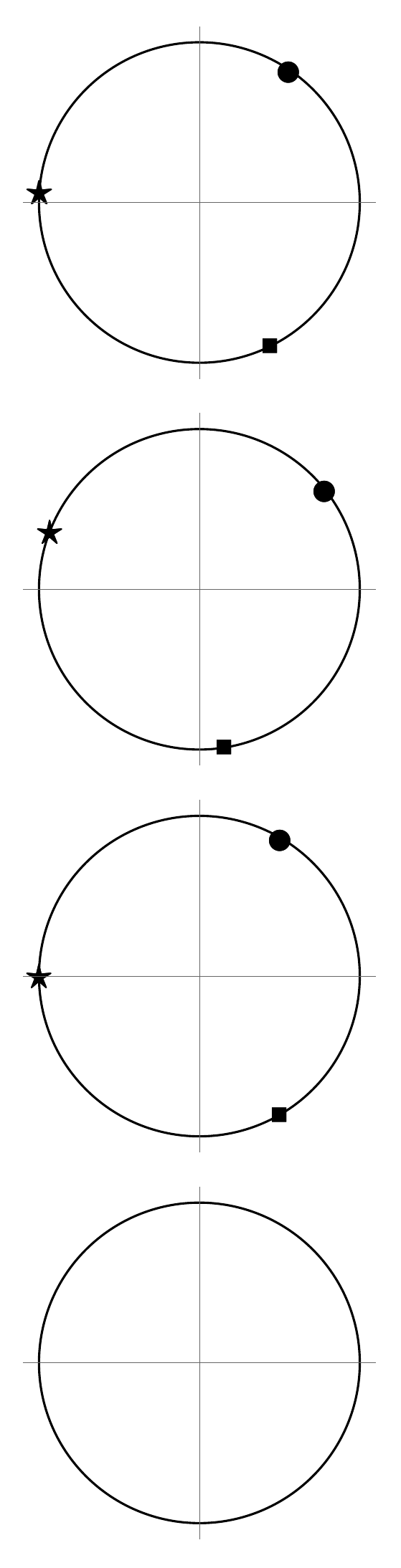}%
   \includegraphics[width=3cm]{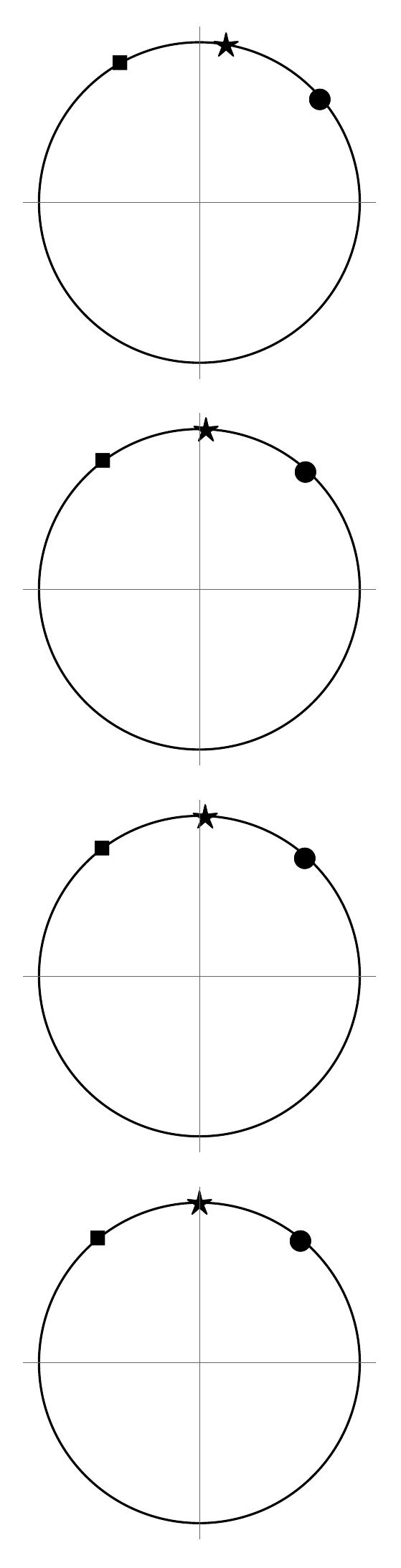}%
    \includegraphics[width=3cm]{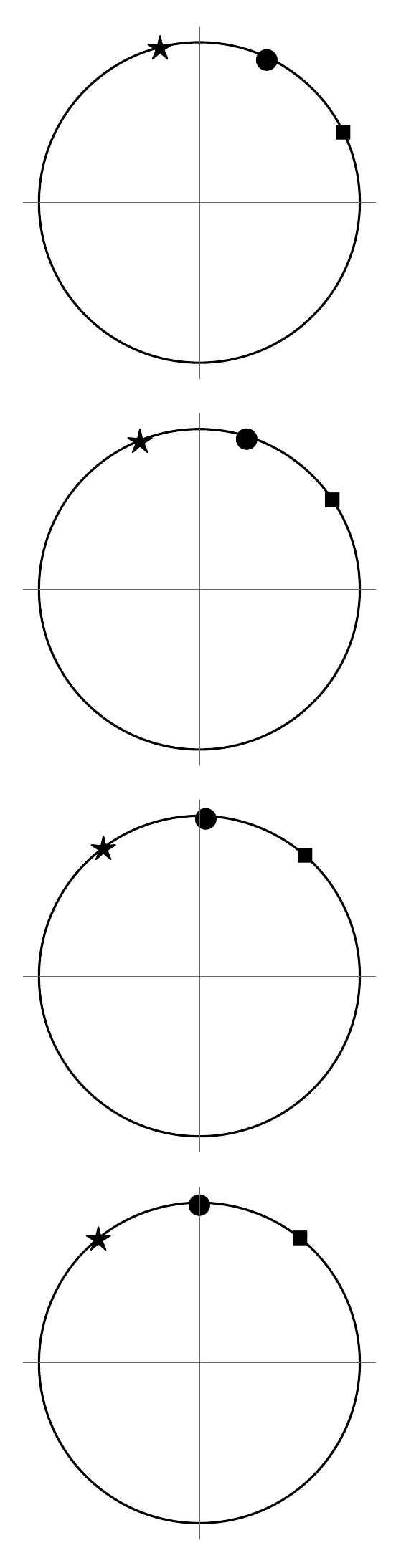}%
    \includegraphics[width=3cm]{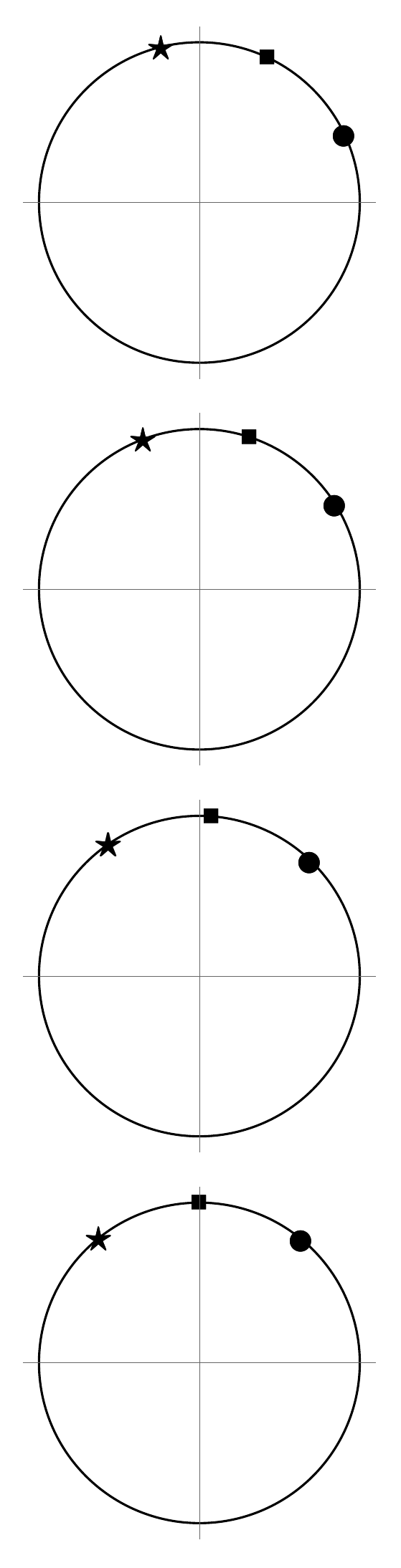}
   \caption{Configurations
   for several masses, from top to bottom,
   $(m_1,m_2,m_3)=(0.1,0.5,1)$,
   $(0.4,0.5,1)$,
   $(0.8,0.9,1)$, and
   $(1,1,1)$.
   The vertical and horizontal axes represent the axis of rotation ($z$-axis)
   and the equator, respectively.
   The masses $m_1$, $m_2$, and $m_3$ are indicated by
   the ball, square, and star, respectively.
   Columns from left to right, 
   the equilateral configuration (left column),
   the isosceles whose centre is $m_3$ (the second column),
   $m_1$ (the third column),
   and $m_2$ (the right).
   The configuration for equilateral of equal masses
   (the bottom left corner) is indefinite,
   any configuration for equilateral shape is $RE$. 
   See section \ref{secConfigEquilateral} for detail.
   }
   \label{figConfigurations}
\end{figure}

\subsection{Configurations of isosceles solution}
For the isosceles shape,
$A>0$,
$s=1$ and $\omega^2=16A/\sqrt{16\sqrt{2}-12}$.
Let $\tau_0=2^{-1}\arccos((\sqrt{2}-1)/2)$
be the arc angle of equal arcs.

For $\theta_2-\theta_3=\theta_3-\theta_1=\tau_0$,
\begin{equation}
A^2=\sum_\ell m_\ell^2
	-m_1m_2\left(2\sqrt{2}-1\right)
	+(m_1+m_2)m_3\left(\sqrt{2}-1\right),
\end{equation}
for $\theta_3-\theta_1=\theta_1-\theta_2=\tau_0$,
\begin{equation}
A^2=\sum_\ell m_\ell^2
	-m_2m_3\left(2\sqrt{2}-1\right)
	+(m_2+m_3)m_1\left(\sqrt{2}-1\right),
\end{equation}
and for 
$\theta_3-\theta_2=\theta_2-\theta_1=\tau_0$,
\begin{equation}
A^2=\sum_\ell m_\ell^2
	-m_3m_1\left(2\sqrt{2}-1\right)
	+(m_3+m_1)m_2\left(\sqrt{2}-1\right).
\end{equation}

In Figure \ref{figConfigurations}, we show 
the configurations for several masses.

\section{Continuation of 
$RE$ shape from
mass independent Euler shape}\label{cont}
In the previous sections,
we were concentrated on the mass independent shapes.
In this section, we search the continuations of 
$RE$ shape near the 
mass independent Euler shapes. We will show that every one of them can be continued to 
mass dependent
Euler shapes.

\subsection{Condition for Euler shapes 
on a rotating meridian}
In this subsection,
we review the condition for Euler shapes
on a rotating meridian.

We have seen that equation \eqref{eqThetaForMeridianForCotangent}
is a necessary and sufficient condition 
for Euler shape.
Let $F_{ij}$ and $G_{ij}$ be
\begin{equation}
\begin{split}
F_{ij}
&=m_im_j\frac{\sin\tau_k}{|\sin\tau_k|^3}
=\frac{m_im_j}{f(\tau_k)},\\
G_{ij}&=m_im_j\sin(2\tau_k),
\end{split}
\end{equation}
for $(i,j,k)=(1,2,3),(2,3,1)$, and $(3,1,2)$,
where $f(x)=\sin(x)|\sin(x)|$.
Then, the condition \eqref{eqThetaForMeridianForCotangent}
for the shape  
$\{\tau_k\}$ 
is equivalent to (see \cite{F-P} for details)
\begin{equation}
d
=\left|\begin{array}{cc}
G_{12}-G_{23} & G_{31}-G_{12} \\
F_{12}-F_{23} & F_{31}-F_{12}
\end{array}\right|
=0
\end{equation} 
for $A\ne 0$.
The explicit expression for $d$ is
\begin{equation}
d=\frac{-m_1m_2m_3 \,g}{f(\tau_1)f(\tau_2)f(\tau_3)},
\end{equation}
where 
\begin{equation}
\begin{split}
g=&m_1f(\tau_1)\Big(
	f(\tau_2)\sin(2\tau_2)-f(\tau_3)\sin(2\tau_3)\Big)\\
&+m_2f(\tau_2)\Big(
	f(\tau_3)\sin(2\tau_3)-f(\tau_1)\sin(2\tau_1)\Big)\\
&+m_3f(\tau_3)\Big(
	f(\tau_1)\sin(2\tau_1)-f(\tau_2)\sin(2\tau_2)\Big).
\end{split}
\end{equation}

Any solution of $g=0$ in $(\tau_1,\tau_2)\in U_\textrm{phys}$
with $A\ne 0$ is an Euler shape.
Therefore, the condition $g=0$ defines one dimensional continuation
of Euler shape in the shape space $U_\textrm{phys}$.

\subsection{Euler shapes near the  equilateral Euler shape}
For Euler shapes near the equilateral solution
$p_\textrm{I}=(2\pi/3,2\pi/3)$,
it is sufficient to consider the region
$U=\{(\tau_1,\tau_2)|
\sin(\tau_1)>0,
\sin(\tau_2)>0,
\sin(\tau_1+\tau_2)<0\}\cap U_\textrm{phys}$.

It is easy to verify that $A^2=\sum_{i<j}(m_i-m_j)^2$ at $p_\textrm{I}$.

\begin{proposition}
For not equal masses, two continuations of mass dependent Euler shape
pass through the equilateral Euler shape.
\end{proposition}

\begin{proof}
For not equal masses case,
since at least one mass is different from the other,
$A^2>0$ at  $p_\textrm{I}$.
Since $A$ is a continuous function of $p$, 
we can take a region in $U$ around $p_\textrm{I}$,
where $A\ne 0$,
and thus any solution of $g=0$ in this region gives 
Euler shape.

 At $p_\textrm{I}$,
$g=\partial g/\partial \tau_1=\partial g/\partial \tau_2=0$.
And the Hessian at this point is
\begin{equation}
H
=\left(\begin{array}{cc}
	\partial^2 g/\partial \tau_1^2 &\partial^2 g/(\partial \tau_1\partial \tau_2) \\
	\partial^2 g/(\partial \tau_2\partial \tau_1) & \partial^2 g/\partial \tau_2^2
\end{array}\right)
=\frac{9\sqrt{3}}{4}
	\left(\begin{array}{cc}
	m_3-m_1 & m_2-m_1 \\
	m_2-m_1 & m_2-m_3
\end{array}\right).
\end{equation}
The determinant is given by
\begin{equation}
\det H=-\frac{243}{32}\sum_{i<j} (m_i-m_j)^2.
\end{equation}
 
By the assumption 
that at least two masses are 
different,
$\det H<0$.
Therefore the point $p_\textrm{I}$ is a saddle point.
So, two $g=0$ contours will pass through
this point.
\end{proof}

\begin{proposition}
For equal masses case, three continuations of 
Euler shape pass through the equilateral Euler shape.
\end{proposition}

\begin{proof}
For equal masses case $m_k=m$, 
$A=0$ has  
just
one solution in $U$,
given by $p_\textrm{I}$.
See Corollary \ref{corSolsOfAeqZeroEqualMasses}
in Appendix \ref{secSolOfA}
for a proof.
As shown above, $p_\textrm{I}$ gives 
Euler shape.
Therefore, any solution of $g=0$ in $U$ gives 
Euler shape.
Fortunately for equal mass in $U$,
the function $g$
has the following simple form
\begin{equation}
\begin{split}
g=\frac{m}{2}
	&\Big(3-\cos\tau_1-\cos\tau_2-\cos(\tau_1+\tau_2)\Big)\\
	&\sin(\tau_1-\tau_2)
	\sin(2\tau_1+\tau_2)
	\sin(\tau_1+2\tau_2).
\end{split}
\end{equation}
Since the first term is positive in $U$,
the solution of $g=0$ in  $U$ are
$\tau_1=\tau_2$,
$2\tau_1+\tau_2=2\pi$, or $\tau_1+2\tau_2=2\pi$.
Thus, on the  
$(\tau_1,\tau_2)$ plane,
the three above straight lines pass through the point 
$p_\textrm{I}=(2\pi/3, 2\pi/3)$,
that is, 
the equilateral triangle Euler shape is not isolated.
\end{proof}

\subsection{Euler shape near 
the isosceles mass independent Euler shape}
In this subsection we will show that each one of the three isosceles mass independent Euler shapes
is not isolated, it has a continuation of mass dependent Euler shapes.

\begin{proposition}
Each one of the three mass independent 
isosceles Euler shape has one continuation of
mass dependent Euler shape.
\end{proposition}

\begin{proof} 
It is enough to show a proof for 
the mass independent isosceles Euler shape
given by
 $p_\textrm{II}=(\tau_0,\tau_0)$, with 
$\tau_0=2^{-1}\arccos((\sqrt{2}-1)/2)$.  
The proof for the other two shapes $p_\textrm{III}=(-2\tau_0,\tau_0)$
and $p_\textrm{IV}=(-\tau_0,2\tau_0)$ follows in a similar way using the same redefinition of coordinates as in Region III and Region IV, in the previous section.

We consider the region $U=\{(\tau_1,\tau_2)|
\sin(\tau_1)>0,
\sin(\tau_2)>0,
\sin(\tau_1+\tau_2)>0\}\cap U_\textrm{phys}$ near the point 
$p_\textrm{II}=(\tau_0,\tau_0)$.
 
Since at this point $A>0$, we can find a small region in $U$ where  $A>0$ and therefore any solution of $g=0$ 
gives an Euler shape.

Now, at $p_\textrm{II}$,
$g=0$ and
\begin{equation}
\begin{split}
g_1
&=\frac{\partial g}{\partial \tau_1}
=\frac{1}{8}\left(\left(22-12\sqrt{2}\right)m_1
	-\left(10-\sqrt{2}\right)(m_2+m_3)
	\right),\\
g_2
&=\frac{\partial g}{\partial \tau_2}
=\frac{1}{8}\left(
	\left(10-\sqrt{2}\right)(m_1+m_3)
	-\left(22-12\sqrt{2}\right)m_2
	\right).
\end{split}
\end{equation}
So,
$g_1=g_2=0$ is impossible,
because, the unique solution for $g_1=g_2=0$ is 
$m_1=m_2=-(1+\sqrt{2})m_3<0$.
Therefore, at least one of $g_1$ or $g_2$ is different from zero;
and by the implicit function theorem,
there is a continuation of $g=0$ that passes through
$p_\textrm{II}$.
\end{proof}

\subsection{Numerical calculations for continuation of Euler shapes
}
\begin{figure}
   \centering
   \includegraphics[width=6cm]{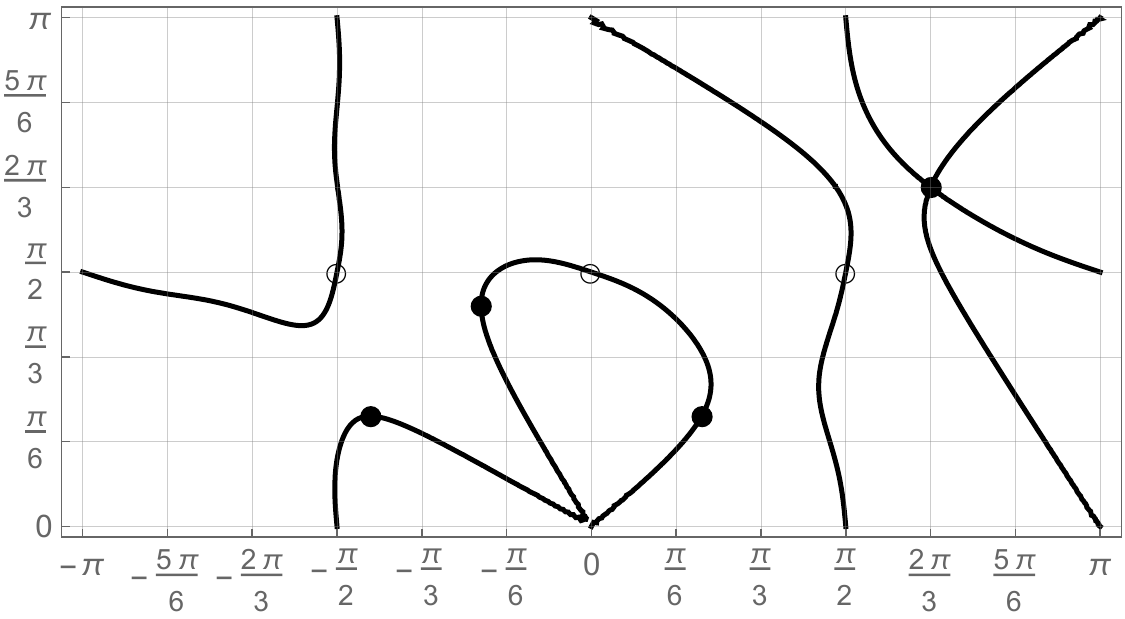}
   \includegraphics[width=6cm]{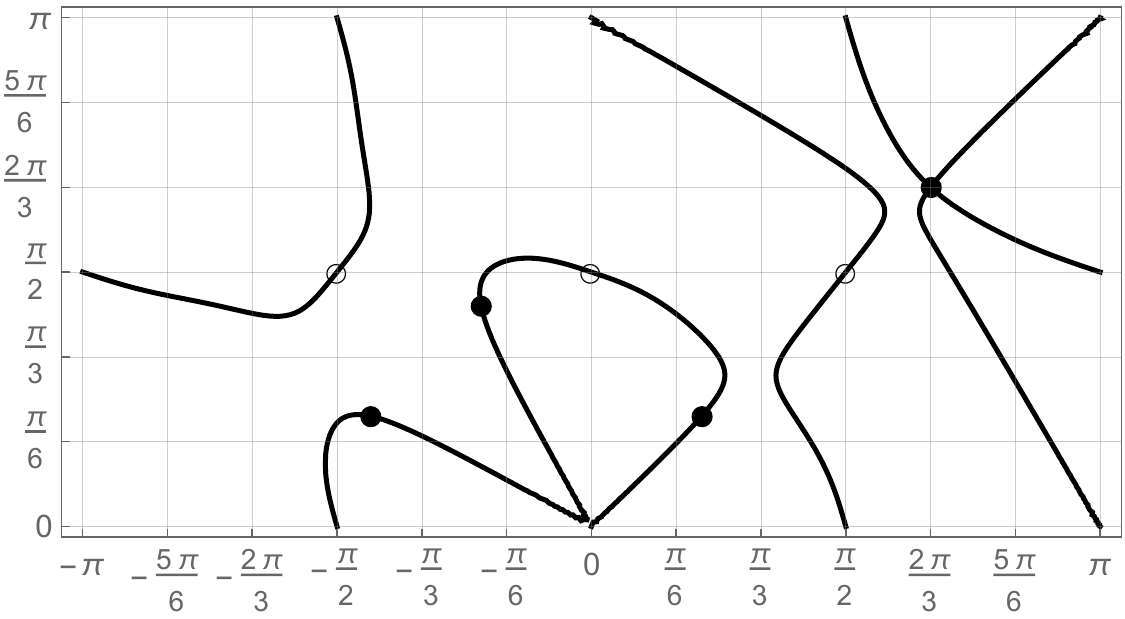}\\
   \includegraphics[width=6cm]{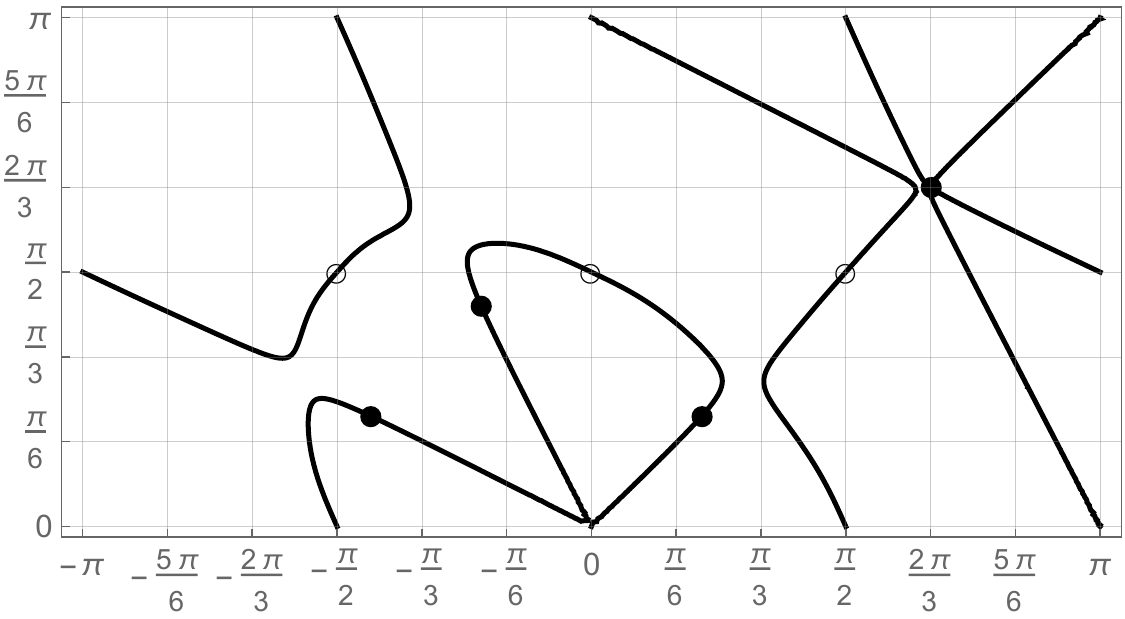}
   \includegraphics[width=6cm]{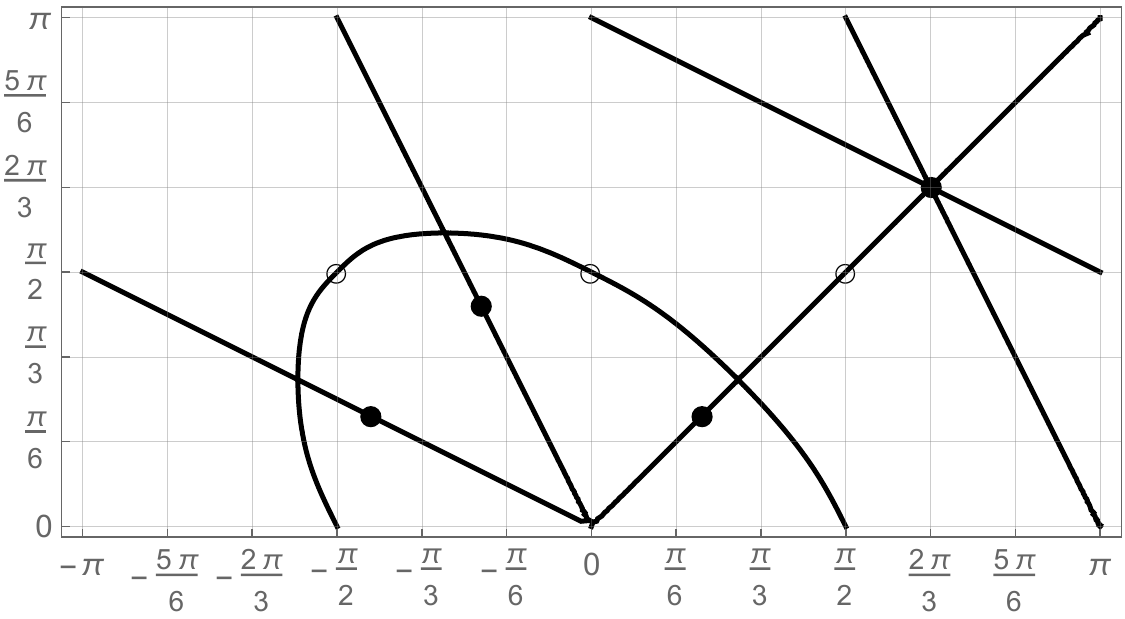} 
   \caption{Continuations of $g(\tau_1,\tau_2)=0$
   in the shape space $U_\textrm{phys}$ for 
   $(m_1,m_2,m_3)=(0.1,0.5,1)$ (upper left),
   $(0.4,0.5,1)$ (upper right), 
   $(0.8,0.9,1)$ (lower left), and
   $(1,1,1)$ (lower right).
   The horizontal and vertical axes are
   $\tau_1=\theta_2-\theta_3$ and $\tau_2=\theta_3-\theta_1$
   respectively.
   In each picture, the
   four black circles represent the mass independent shapes,
   and three hollow circles represent the excluded shapes
   (singular points).
}
   \label{figContinuationsOfERE}
\end{figure}

In Figure \ref{figContinuationsOfERE},
we show continuations of Euler shapes 
for several masses
that are represented by $g=0$.
As you can see,
the continuation curve changes as the masses are changed.
However, as we proved in section \ref{ind-conf}
and shown in Figure \ref{figContinuationsOfERE},
the continuation curve passes through
the (not moved) mass independent Euler shape.

\section{Isosceles or equilateral shapes
when all masses are different}
\label{Isos-eq}

In this section, we consider $RE$ shapes
in the case
$m_i \ne m_j$ for all $i\ne j$,  
which can be or not mass independent. We have the following result.

\begin{proposition}
When all masses are different, the
 isosceles and the equilateral $RE$ shapes are the unique two mass independent Euler shapes on a rotating meridian.
\end{proposition}
\begin{proof}
For  Lagrange shapes,
let  
$\sigma_{23}=\sigma_{31}=\sigma$
be the equal arcs.
Then, the condition $\lambda_1=\lambda_2$ 
in \eqref{eqForLagrangeII} for Lagrange shape requires 
\begin{equation}
(m_1-m_2)(\cos\sigma_{23}-1)=0.
\end{equation}
But this is impossible for $m_1\ne m_2$.

For Euler isosceles triangle shape on the equator,
let $\phi_2-\phi_3=\phi_3-\phi_1$.
Then the condition \eqref{eqForEREOntheEquator}
requires
$(m_1-m_2)m_3=0$.
But this is impossible.

For Euler isosceles triangle shape on a rotating meridian
with $A\ne 0$,
let $\theta_2-\theta_3=\theta_3-\theta_1=\tau$.
Then the condition \eqref{eqThetaForMeridianForCotangent} requires
\begin{equation}
\begin{split}
-&m_1m_2\left(\frac{s\omega^2}{2A}\sin(4\tau)
- \frac{\sin(2\tau)}{|\sin(2\tau)|^3}
\right)\\
=&m_2m_3\left(\frac{s\omega^2}{2A}\sin(2\tau)
- \frac{\sin\tau}{|\sin\tau|^3}
\right)\\
=&m_3m_1\left(\frac{s\omega^2}{2A}\sin(2\tau)
- \frac{\sin\tau}{|\sin\tau|^3}
\right).
\end{split}
\end{equation}
Since $m_1\ne m_2$,
the last two lines require that the term inside 
the parentheses is zero.
Then, the term in the first line must be zero, 
which are the conditions for the mass independent shapes
on a rotating meridian.

For a shape with $A=0$,
$\theta_2-\theta_3=\theta_3-\theta_1=\tau$
requires $m_1=m_2$
as shown in  Corollary \ref{AeqZeroAndIsosceles} (see
Appendix A in Section \ref{secSolOfA}).
\end{proof}

\section{Appendix A. Solutions for $A=0$}\label{secSolOfA}
In this section,
the solutions of $A=0$ in $U_\textrm{phys}$
are described,
where $A$ and $U_\textrm{phys}$
are defined by \eqref{defOfA}
and \eqref{defOfUphys}.

\begin{proposition}\label{solOfAeq0}
The solutions of $A=0$ in $U_\textrm{phys}$ are given by,
\begin{equation}\label{cosSinOf2DeltaTheta}
\begin{split}
\cos(2\tau_k)
&=\frac{m_k^2-m_i^2-m_j^2}{2m_im_j},\\
\frac{\sin(2\tau_3)}{m_3}
&=\frac{\sin(2\tau_1)}{m_1}
=\frac{\sin(2\tau_2)}{m_2},
\end{split}
\end{equation}
where
$\tau_k=\theta_i-\theta_j$,
$(i,j,k)=(1,2,3),(2,3,1)$, and $(3,1,2)$;
the masses must satisfy the triangle inequality $m_i+m_j>m_k$
for all choices of $(i,j,k)$,
otherwise there are no solutions.
\end{proposition}
\begin{proof}
Since, 
$A^2=\left|\sum_\ell m_\ell e_\ell \right|^2$
with 
$e_\ell = (\cos(2\theta_\ell),\sin(2\theta_\ell))\in \mathbb{R}^2$,
$A=0$ is equivalent to $\sum_\ell m_\ell e_\ell=0$.
This means that
the three vectors $m_\ell e_\ell$ form a triangle
with sides of length $m_\ell$.

Therefore $\{m_k\}$ must satisfy the triangle inequality.
The equality is excluded,
because otherwise at least one of the equations should satisfy
$2\tau_k\equiv 0 \mod 2\pi$
(namely $\tau_k \equiv 0 \mod \pi$)
which is excluded in $U_\textrm{phys}$.

Using $m_k e_k =-(m_i e_i+m_j e_j)$,
$m_k^2=\left|m_i e_i + m_j e_j\right|^2$ and
$0=(m_i e_i+m_j e_j)\times e_k$
we obtain
the equations in \eqref{cosSinOf2DeltaTheta}.
\end{proof}

\begin{cor}
The number of solutions of $A=0$ in $U_\textrm{phys}$
for given $\{m_k\}$ could be $0$ or $4$.
\end{cor}
\begin{proof}
If the masses do not satisfy the triangle inequality,
the number of solutions is zero.

For the masses that satisfy the triangle inequality,
by 
the first equation of \eqref{cosSinOf2DeltaTheta},
\begin{equation}\label{eqForCosTau}
\cos\tau_k
=\pm \sqrt{\frac{m_k^2-(m_i-m_j)^2}{4m_im_j}}
\,\ne 0.
\end{equation}

Now, let
\begin{equation}
\alpha_k=\arccos\sqrt{\frac{m_k^2-(m_i-m_j)^2}{4m_im_j}}
\in (0,\pi/2)
\mbox{ for } k=1,2,
\end{equation}
be one of the solution of \eqref{eqForCosTau}.
Then the solutions of this equation are
\begin{equation}
\begin{split}
\tau_1&=
-\pi+\alpha_1,\quad
-\alpha_1, \quad
\alpha_1, \quad
\pi-\alpha_1
	\in (-\pi,\pi),\\
\tau_2&=\alpha_2, \quad \pi-\alpha_2 \in (0,\pi).
\end{split}
\end{equation}
By the second line of \eqref{cosSinOf2DeltaTheta},
$\sin(2\tau_1)$ and $\sin(2\tau_2)$ must have the
same sign,
therefore, the solutions of \eqref{cosSinOf2DeltaTheta}
are the following four,
\begin{equation}
(\tau_1,\tau_2)
=
(-\pi+\alpha_1,\alpha_2), \quad
(-\alpha_1,\pi-\alpha_2), \quad
(\alpha_1,\alpha_2), \quad
(\pi-\alpha_1,\pi-\alpha_2).
\end{equation}
\end{proof}

\begin{cor}\label{corSolsOfAeqZeroEqualMasses}
For equal masses case,
the four solutions of $A=0$ in $U_\textrm{phys}$
are
$(\tau_1,\tau_2)=(-2\pi/3,\pi/3)$, \,\,
$(-\pi/3,2\pi/3)$, \,\,
$(\pi/3,\pi/3)$, \,\,
$(2\pi/3,2\pi/3)$.
\end{cor}
\begin{proof}
For the equal masses case,
$\cos\tau_k=\pm 1/2$,
therefore $\alpha_1=\alpha_2=\pi/3$.
Then, the solutions are obviously
the four given above.
\end{proof}

\begin{cor}\label{AeqZeroAndIsosceles}
The shape $A=0$ and $\tau_i=\tau_j$
in $U_\textrm{phys}$
is realised only when $m_i=m_j$.
\end{cor}
\begin{proof}
It is obvious by the 
equation \eqref{cosSinOf2DeltaTheta}.
\end{proof}

\subsection*{Acknowledgements} The second author (EPC) has been partially supported by Asociación Mexicana de Cultura A.C.

\end{document}